\documentclass[12pt,a4paper]{amsart}
\usepackage{latexsym,amsfonts,amsthm,amsmath,mathrsfs,amssymb}
\usepackage[british]{babel}
\usepackage[dvips]{color}
\usepackage[utf8]{inputenc}
\usepackage[T1]{fontenc}
\usepackage[dvips,final]{graphics}
\usepackage{xcolor}
\usepackage{mathrsfs}
\usepackage{mathtools} 
\usepackage{breqn}
\usepackage{epstopdf}
\usepackage{verbatim}
\usepackage{amscd}
\usepackage{hyperref}
\usepackage[english,capitalize]{cleveref}
\usepackage{aliascnt}

\usepackage{todonotes}

\makeatletter
\def\newaliasedtheorem#1[#2]#3{
	\newaliascnt{#1@alt}{#2}
	\newtheorem{#1}[#1@alt]{#3}
	\expandafter\newcommand\csname #1@altname\endcsname{#3}
}
\makeatother

\theoremstyle{plain}

\newtheorem{theorem}{Theorem}[section]
\newaliasedtheorem{prop}[theorem]{Proposition}
\newaliasedtheorem{lem}[theorem]{Lemma}
\newaliasedtheorem{coroll}[theorem]{Corollary}

\theoremstyle{definition}
\newaliasedtheorem{defi}[theorem]{Definition}
\newaliasedtheorem{quest}[theorem]{Question}
\newaliasedtheorem{fact}[theorem]{Fact}

\theoremstyle{remark}
\newaliasedtheorem{rem}[theorem]{Remark}
\newaliasedtheorem{exa}[theorem]{Example}

\newcommand{\R}{\mathbb{R}}

\newcommand{\LL}{\mathbb{L}}

\newcommand{\C}{\mathbb{C}}

\let\altphi\phi
\let\phi\varphi
\let\varphi\altphi
\let\altphi\undefined

\newcommand{\average}{{\mathchoice {\kern1ex\vcenter{\hrule height.4pt
width 6pt
depth0pt} \kern-9.7pt} {\kern1ex\vcenter{\hrule height.4pt width 4.3pt
depth0pt}
\kern-7pt} {} {} }}

\address{\textsc{Daniela Di Donato}: 
Dipartimento di Ingegneria Industriale e Scienze Matematiche, Via Brecce Bianche, 12 60131 Ancona, Universit\'a Politecnica delle Marche.}
\email{d.didonato@staff.univpm.it}

\title{Normed spaces using intrinsically Lipschitz sections and Extension Theorem for the intrinsically H\"older sections}

\date{\today}

\author{ Daniela Di Donato}

\email{}

\begin{document}

\begin{abstract}

The purpose of this article is twofold: first of all, we want to define two norms using the space of intrinsically Lipschitz sections. On the other hand, we want to generalize an Extension Theorem proved by the author in the context of the intrinsically H\"older sections with target a topological space $Y.$ Here our target will be $Y\times \R^s$ with $s \geq 1$ instead of $Y.$

		\end{abstract}

\maketitle 
\tableofcontents

\section{Introduction}
In this paper, we focus our attention on a new point of view for the intrinsically Lipschitz graphs in the Franchi-Serapioni-SerraCassano sense  \cite{FSSC, FSSC03, MR2032504}  (see also  \cite{SC16, FS16}) in metric spaces.

They introduced and analized this notion in order to establish a good notion of rectifiability in a particular metric spaces called subRiemannian Carnot groups \cite{ABB, BLU, CDPT}.

In the usual way, Federer \cite{Federer} says that a subset of $\R^n$ is countably $d$-rectifiable if it is possible to cover it with a countable union of suitable graphs. More precisely, he considers graphs of Lipschitz maps $f:\R^d \to \R^{n-d}.$ However, Ambrosio and Kirchheim \cite{AmbrosioKirchheimRect} (see also \cite{MAGNANI1}) proved that this definition does not work in Carnot groups and so many mathematicians give other notions of rectifiability. The reader can see \cite{AM2022.1, AM2022.2, DB22, DS91a, DS91b, MR2247905, MR2048183, NY18}. As we said, another possible solution is given by  Franchi-Serapioni-SerraCassano with  the so-called "Intrinsic Lipschitz maps" in Carnot groups. The idea is similar to Euclidean case: firstly, they introduce suitable cones called intrinsic cones which are not equivalent with the Euclidean ones; and, then, they say that a map $\phi$ is intrinsic Lipschitz if it is possible to have for any point $p$ belongs to the graph of $\phi$ an empty intersection between a suitable intrinsic cone with vertex $p$ and the graph of this map.

Recently, Le Donne and the author generalize this concept in metric spaces (see \cite{DDLD21}). 

A basic difference is the following: in Franchi, Serapioni and Serra Cassano approach, they consider intrinsically Lipschitz maps. On the other hand, in our approach we consider the graphs and this a bit change is so important because:
 \begin{itemize}
\item The setting are more general: the class of the metric spaces is larger than the class of Carnot groups.
\item The proofs are elegantly short and simple.
\item We use basic mathematical tools in the proofs.
\end{itemize}

In a natural way, in \cite{D22.1} the author introduce the notion of intrinsically H\"older sections which extend the Lipschitz ones. Here, 	our setting is the following. We have a metric space $X$, a topological space $Y$, and a 
quotient map $\pi:X\to Y$, meaning
continuous, open, and surjective.
The standard example for us is when $X$ is a metric Lie group $G$ (meaning that the Lie group $G$ is equipped with a left-invariant distance that induces the manifold topology), for example a subRiemannian Carnot group, 
and $Y$ is the space of left cosets $G/H$, where 
$H<G$ is a  closed subgroup and $\pi:G\to G/H$ is the projection modulo $H$, $g\mapsto gH$.

\begin{defi}[Intrinsic H\"older section]\label{Intrinsic H\"older  section}
Let $(X,d)$ be a metric space and let $Y$ be a topological space. We say that a map $\phi :Y \to X$ is a {\em section} of a quotient map $\pi :X \to Y$ if
\begin{equation*}
\pi \circ \phi =\mbox{id}_Y.
\end{equation*}
Moreover, we say that $\phi$ is an {\em intrinsically $(L, \alpha)$-H\"older  section} with constant $L>0$ and $\alpha \in (0,1)$ if in addition
\begin{equation}\label{Intrinsic H\"older per semplificare}
d(\phi (y_1), \phi (y_2)) \leq L d(\phi (y_1), \pi ^{-1} (y_2))^\alpha +  d(\phi (y_1), \pi ^{-1} (y_2)), \quad \mbox{for all } y_1, y_2 \in Y.
\end{equation}

Equivalently, we are requesting that
\begin{equation*}
d(x_1, x_2) \leq L d(x_1, \pi ^{-1} (\pi (x_2)))^\alpha + d(x_1, \pi ^{-1} (\pi (x_2))), \quad \mbox{for all } x_1,x_2 \in \phi (Y) .
\end{equation*}
\end{defi}

When $\alpha =1,$ a section $\phi$ is intrinsic Lipschitz in the sense of \cite{DDLD21}; and, if in addition, $\pi$ is a Lipschitz quotient or submetry \cite{MR1736929, Berestovski}, the results trivialize, since in this case being intrinsically Lipschitz  is equivalent to biLipschitz embedding, see Proposition 2.4 in \cite{DDLD21}. In a natural way, following the seminal papers \cite{AGS08.2, MR2480619, Savar1} (see also \cite{AGS08.4, AGS08.3, Savar2, S06, MR2459454}), the author introduced and studied the link between the intrinsic sections/intrinsic Lipschitz sections and the intrinsic Hopf-Lax semigroups \cite{D22.NOV26A, D22.NOV26B}.

The purpose of this article is twofold: first of all, we want to define two norms using the notion of Lipschitz sections. Second, we want to generalize an Extension Theorem with target $Y$ which is a topological space; in this paper, our target will be $Y \times \R^s$ instead of $Y.$

More precisely, in Section \ref{Normed space for the intrinsic continuous sections}, the main results are Theorem \ref{24theoremNormed.1} and \ref{24theoremNormed.2}. Here, we define two possible normed spaces using the following ingredients:
\begin{itemize}
\item we know that there is a large class of intrinsically H\"older sections and so Lipschitz sections that is a vector space over $\R$ or $\C$ (see Theorem \ref{theorem24});
\item we can define two different norms noting the following simple fact: in the usual case, we have that $d(x,y) =d(y,x)$ for any point $x,y \in X;$ on the other hand, in our intrinsic context, in general, we have that $d(f(x),\pi^{-1} (y)) \ne d(f(y),\pi^{-1} (x)),$ for every $x,y \in X.$ 
\item we obtain the homogeneity of our norms defined in \eqref{banachspace} and in \eqref{banachspace.18}  thanks to linearity of $\pi$ and, in particular, to Lemma \ref{lem22a}.
\end{itemize}

Finally, in Section \ref{Level sets and extensions} the main result is Theorem \ref{thm3.Rkappa} which generalizes Extension Theorem for the intrinsically H\"older sections proved in \cite[Theorem 1.3]{D22.1}. The main difference is that, in this project, the target space is a topological space $Y\times \R^s$ with $s\geq 1$ instead of $Y.$ As in Vittone's case, we build each component $f_i$ for $i=1,\dots , s$ separately and then join them without any additional assumptions. However, the final step when the target space is only $Y$ does not provide a Lipschitz map $f=(f_1,\dots , f_s).$


\section{Intrinsically  H\"older  sections}
\label{sec:equiv_def}

 \subsection{Intrinsically  H\"older  sections: when $Y$ is bounded}
Definition \ref{Intrinsic H\"older  section} it is very natural if we think that what we are studying graphs of appropriate maps. However, in the following proposition, we introduce an equivalent condition of \eqref{Intrinsic H\"older per semplificare} when $Y$ is a bounded set.

\begin{prop}\label{Intrinsic H\"older  section.2}
Let $\pi: X \to Y$ be a quotient map between a metric space $X$ and a topological and bounded space $Y$ and let $\alpha \in (0,1).$ The following are equivalent:
\begin{enumerate}
\item there is $L >0$ such that
\begin{equation*}
d(\phi (y_1), \phi (y_2)) \leq L d(\phi (y_1), \pi ^{-1} (y_2))^\alpha +  d(\phi (y_1), \pi ^{-1} (y_2)), \quad \mbox{for all } y_1, y_2 \in Y.
\end{equation*}
\item there is $K \geq 1$ such that
\begin{equation}\label{equationEQUIV}
d(\phi (y_1), \phi (y_2)) \leq K d(\phi (y_1), \pi ^{-1} (y_2))^\alpha , \quad \mbox{for all } y_1, y_2 \in Y.
\end{equation}
\end{enumerate}
\end{prop}

We further rephrase the definition as saying that $\phi(Y)$, which we call the {\em graph} of $\phi$, avoids some particular sets (which depend on $\alpha , L$ and $\phi$ itself):


\begin{prop}\label{propo_ovvia} Let $\pi :X \to Y$  be a  quotient map between a metric space and a topological space, $\phi: Y\to X$ be a section of $\pi$, $\alpha \in (0,1)$ and $L>0$.
Then $\phi$ is intrinsically $(L, \alpha)$-H\"older if and only if
\begin{equation*}
 \phi  (Y) \cap  R_{x,L} = \emptyset , \quad \mbox{for all } x \in \phi (Y),
\end{equation*}
where $$R_{x,L} := \left\{ x'\in X \;|\;   L d(x', \pi ^{-1} (\pi (x)))^\alpha + d(x', \pi ^{-1} (\pi (x)))  <   d(x', x)\right\}.$$


%
\end{prop}

Proposition \ref{propo_ovvia} is a triviality, still its purpose is to stress the analogy with the intrinsically Lipschitz sections theory introduced in \cite{DDLD21} when $\alpha =1$.  In particular, the sets $R_{x,L}$ are the intrinsic cones in the sense of  Franchi, Serapioni and Serra Cassano when $X$ is a  subRiemannian Carnot group and $\alpha =1.$ The reader can see \cite{D22.3} for a good notion of intrinsic cones in metric groups.

 \subsection{Equivalent definition for intrinsic H\"older sections}
\begin{defi}[Intrinsically H\"older  with respect to  a section]\label{defwrtpsinew}
 Given  sections 
  $\phi, \psi :Y\to X$   
  of $\pi$. We say that   $\phi $ is {\em intrinsically $(L, \alpha)$-H\"older with respect to  $\psi$ at point $\hat x$}, with $L>0, \alpha \in (0,1)$ and $\hat x\in X$, if
\begin{enumerate}
\item $\hat x\in \psi(Y)\cap \phi (Y);$
\item $\phi  (Y) \cap  C_{\hat x,L}^{\psi} = \emptyset ,$
\end{enumerate}
where
$$ C_{\hat x,L}^{\psi} := \{x\in X \,:\,  d(x, \psi (\pi (x))) > L d(\hat x, \psi (\pi (x)))^\alpha + d(\hat x, \psi (\pi (x))) \}.   $$
\end{defi}

  \begin{rem} Definition~\ref{defwrtpsinew} can be rephrased as follows.
 A section $\phi  $ is intrinsically $(L,\alpha)$-H\"older with respect to  $\psi$ at point $\hat x$ if
 and only if 
 there is $\hat y\in Y$ such that  $\hat x= \phi (\hat y)=\psi(\hat y)$ and
\begin{equation}\label{defintrlipnuova}
 d(x, \psi (\pi (x))) \leq L d(\hat x, \psi (\pi ( x)))^\alpha + d(\hat x, \psi (\pi (x))), \quad \forall x \in \phi (Y), 
\end{equation}
which equivalently means 
\begin{equation}\label{equation28.0}
 d(\phi (y), \psi (y)) \leq L d(\psi(\hat y), \psi (y))^\alpha +d(\psi(\hat y), \psi (y)) ,\qquad \forall y\in Y. 
\end{equation}
  \end{rem}  

Notice that Definition \ref{defwrtpsinew} does not induce an equivalence relation because of lack of symmetry in the right-hand side of \eqref{equation28.0}. However, following Cheeger theory \cite{C99} (see also \cite{K04, KM16}), in \cite[Theorem 4.2]{D22.1} we give an  equivalent property of H\"older section. 

Being intrinsically Lipschitz section is equivalent to the last definition as proved in  \cite[Proposition 1.5]{D22.1}
 \begin{prop}\label{linkintrinsicocneelip}
   Let $X  $ be a metric space, $Y$ a topological and bounded space, $\pi :X \to Y$   a  quotient map, $L\geq 1$ and $\alpha , \beta , \gamma \in (0,1)$.
Assume that   every point $x\in X$ is contained in the image of an intrinsic $(L,\alpha )$-H\"older section $\psi_x$ for $\pi$.
 Then for every section $\phi :Y\to X$ of $\pi$ the following are equivalent:
   \begin{enumerate}
\item for all $x\in\phi(Y)$ the section $\phi $ is intrinsically $(L_1,\beta )$-H\"older with respect to  $\psi_x$ at   $x;$
\item  the section $\phi $  is intrinsically $(L_2,\gamma)$-H\"older.
\end{enumerate}
   \end{prop}
   
   We conclude this section recall an important concept which we will be used later.

       \begin{defi}[Intrinsic H\"older set with respect to $\psi$]\label{defwrtpsinew.9apr.23} Let $\alpha \in (0,1]$ and $\psi: Y \to X$ a section of $\pi$.  We define the set  of all  intrinsically H\"older  sections  of $\pi$ with respect to  $\psi$ at point $\hat x$ as
\begin{equation*}
\begin{aligned}
H _{\psi , \hat x, \alpha} & :=\{ \phi :Y\to X \mbox{ a section of $\pi$}  \, :\, \phi \mbox{ is intrinsically $(\tilde L, \alpha )$-H\"older w.r.t. $\psi$ at point $\hat x$} \\
& \quad \quad \mbox{  for some $\tilde L>0$} \}.
 \end{aligned}
\end{equation*}
  \end{defi}
  
%
%
   
In particular, in \cite{D22.1} we have the following statement regarding  the set $H _{\psi , \hat x, \alpha} $.
  
%

   \begin{theorem}[Theorem 3.5 \cite{D22.1}]\label{theorem24} Let $\pi :X \to Y$ is a linear and quotient map from a normed space $X$ to a topological space $Y.$ Assume also that $\psi :Y \to X$ is a section of $\pi$  and $\{\lambda \hat x\, :\, \lambda \in \R^+\} \subset X$ with $\hat x \in \psi (Y).$
   
Then, for any $\alpha \in (0,1],$ the set $\bigcup _{\lambda \in \R^+} H _{\lambda \psi , \lambda \hat x , \alpha} \cup \{ 0\}$ is a vector space over $\R $ or $\C .$
       \end{theorem}
       
         Notice that  it is no possible to obtain the statement for $H_{\psi , \hat x, \alpha }$ since the simply observation that if $\psi (\hat y) = \hat x$ then $\psi (\hat y ) + \psi (\hat y )  \ne \hat x.$ 
%
%
%

  \section{Normed space for the intrinsically Lipschitz sections}\label{Normed space for the intrinsic continuous sections}
    \subsection{Normed space: Version 1}
  In this section, we consider the case of intrinsically Lipschitz sections, i.e., $\alpha =1.$ 
  
    Let $\pi :\R ^\kappa \to Y$ be a quotient map with $Y\subset \R^\kappa $. Assume also that $K \subset Y$ is a compact set and $\psi _{|K} :K \to \R$ is an intrinsically $L$-Lipschitz section of $\pi$ with $L \geq 1$ and $\hat x = \psi (\bar y) \in \psi (Y).$   We will use the following notation $$ILS _{\lambda \psi _{|K},\lambda \hat x} := H  _{\lambda \psi _{|K},\lambda \hat x, 1 }.$$ 
  Here, we show that $$ (\LL , \|.\|):= \left( \bigcup_{\lambda \in \R^+} ILS _{\lambda \psi _{|K},\lambda \hat x}  \cup \{ 0\},\|.\| \right)$$ is a normed space for a suitable norm $\|.\|=\|. \|_{ ILS_{\lambda \psi , \lambda \hat x}}: \LL \to \R^+$  defined as for any $\phi \in \LL,$
   \begin{equation}\label{banachspace}
\|\phi \|_{ ILS_{\lambda \psi , \lambda \hat x}} := \|\phi\|_\infty + [\phi]_{\lambda \psi , \lambda \hat x},
\end{equation}
where $\|\phi \|_\infty := \sup_{y\in K} |\phi (y)|$ and 
\begin{equation*}
 [\phi]_{\lambda \psi , \lambda \hat x}:= \sup_{y \in K}  d(\lambda \phi (y),(1/ \lambda \pi )^{-1} (\pi(\hat x)))  .
\end{equation*}

Then, we are able to give the main result of this section.

   \begin{theorem}\label{24theoremNormed.1}
   Let $\pi :\R^\kappa \to Y$ be a linear and quotient map with $Y\subset \R^\kappa $. Assume also that $\psi :K \to \R^\kappa $ is an intrinsically $L$-Lipschitz section of $\pi$ with $K \subset Y $ compact, $L \geq 1$ and $\hat x \in X.$ Then, the set $ \LL$ endowed with  $\|\cdot \|_{ ILS_{\psi , \hat x}}$    is a normed space.
       \end{theorem}

   We need the following lemma.
   \begin{lem}[Lemma 4.7 \cite{D22.31may}]\label{lem22a}
   Let $X$ be a normed space, $Y$ be a topological space and $\pi: X \to Y$ be a linear and quotient map. Then
\begin{equation}\label{equ22may.1}
| \lambda | d(x_1,\pi ^{-1} (y_2) ) =  d(\lambda x_1, (1/\lambda  \pi )^{-1} (y_2) ), \quad \forall x_1\in \R^\kappa , y_2 \in Y, \lambda \in \R -\{0\}.
\end{equation}
   \end{lem}

  \begin{rem}
An easy corollary of Lemma \ref{lem22a} when $Y\subset \R$ and $X=\R^\kappa$ is that
  \begin{equation*}
   \begin{aligned}
\lim_{h\to 0^+} \frac{d(h\phi (t+h), (1/h \pi)^{-1} (t)))}{h} & =0,\\
\lim_{h\to 0^+} \frac{d(h\phi (t), (1/h \pi)^{-1} (t+h)))}{h} & =0,\\
\end{aligned}
\end{equation*}
for $t\in Y.$ Indeed, for $h>0$
  \begin{equation*}
   \frac{d(h\phi (t+h), (1/h \pi)^{-1} (t))}{h} = d(\phi (t+h),  \pi^{-1} (t)) \leq  d(\phi (t+h), \phi (t) ),
\end{equation*}
and so take to the limit for $h\to 0,$ we obtain the first limit thank to the continuity of $\phi$. In a similar way, it is possible to see the second limit.
   \end{rem}

At this point, we give the proof of Theorem \ref{24theoremNormed.1}.
\begin{proof}[Proof of Theorem \ref{24theoremNormed.1}]
The fact $\|\phi\| \equiv 0$ if and only if $\phi \equiv 0$ follows because $\|.\|_\infty $ is a norm. On the other hand, since $\pi $ is linear map  and thanks to Lemma \ref{lem22a}, we have $$\sup_{y \in K}  d(\delta \phi(y),(1/\delta \pi)^{-1} (\pi(\hat x))) = \sup_{y \in K} |\delta |  d(\phi (y), \pi^{-1} (\pi(\hat x)))$$ for any $\delta \in \R- \{ 0\}$ and so
   \begin{equation}\label{banachspace.25}
\|\delta \phi\|_\infty + [\delta \phi]_{\psi , \hat x} = |\delta | ( \|\phi\|_\infty + [\phi]_{\psi , \hat x}),
\end{equation}
for any $\phi \in \LL.$ 

Finally, the triangle inequality follows using again the linearity of $\pi.$ Indeed, if $\phi , \eta \in  \LL - \{ 0\} $ and, in particular, $\phi , \eta \in   ILS _{(\lambda _1+\lambda _2) \psi _{|K},(\lambda _1+\lambda _2) \hat x}$ then for $x_\phi , x_\eta \in \R^\kappa $ such that 
\begin{equation*}
\begin{aligned}
 [ \phi  ]_{(\lambda _1+\lambda _2) \psi , (\lambda _1+\lambda _2)  \hat x} &= \sup_{y \in K}  d( (\lambda _1 + \lambda _2 )  \phi (y) ,(1/ (\lambda _1 + \lambda _2 ) \pi )^{-1} (\pi(\hat x))) = d((\lambda _1 + \lambda _2 ) \phi (y), x_\phi ) \\
[\eta  ]_{ (\lambda _1+\lambda _2) \psi , (\lambda _1+\lambda _2) \hat x} & = \sup_{y \in K}  d( (\lambda _1 + \lambda _2 )  \eta (y),(1/(\lambda _1 + \lambda _2 ) \pi)^{-1} (\pi(\hat x))) = d((\lambda _1 + \lambda _2 ) \eta (y), x_\eta )  \\
\end{aligned}
\end{equation*}
one gets
\begin{equation*}
\begin{aligned}
 [ \phi +\eta ]_{(\lambda _1 + \lambda _2 )\psi ,(\lambda _1 + \lambda _2 ) \hat x} &= \sup_{y \in K}  d((\lambda _1 + \lambda _2 ) (\phi (y) + \eta (y)),(2/ (\lambda _1 +\lambda _2) \pi)^{-1} (\pi(\hat x))) \\
 & \leq \| (\lambda _1 + \lambda _2 )  \phi (y) + (\lambda _1 + \lambda _2 )  \eta (y) -(x_\phi + x_\eta )\| \\
 &\leq  \| (\lambda _1 + \lambda _2 )  \phi (y) -x_\phi \| + \| (\lambda _1 + \lambda _2 )  \eta (y) - x_\eta \|,\\
 &  [ \phi  ]_{(\lambda _1 + \lambda _2 ) \psi ,(\lambda _1 + \lambda _2 ) \hat x}  + [\eta  ]_{ (\lambda _1 + \lambda _2 ) \psi , (\lambda _1 + \lambda _2 )  \hat x},
\end{aligned}
\end{equation*}
where in the first equality, by linearity of $\pi,$ we used the fact 
\begin{equation*}
\begin{aligned}
\pi\big((\lambda _1 + \lambda _2 ) (\phi (y) + \eta (y)) \big) & = \pi \big((\lambda _1 + \lambda _2 ) \phi (y)  \big) +\pi \big((\lambda _1 + \lambda _2 ) \eta (y) \big) \\
&= (\lambda _1 + \lambda _2 ) (\pi(\phi (y)) + (\pi(\eta (y)) )\\
&= 2(\lambda _1 + \lambda _2 )y.
\end{aligned}
\end{equation*}
Hence, $$  [ \phi +\eta ]_{(\lambda _1 + \lambda _2 )\psi ,(\lambda _1 + \lambda _2 ) \hat x}  \leq   [ \phi  ]_{(\lambda _1 + \lambda _2 ) \psi ,(\lambda _1 + \lambda _2 ) \hat x}  + [\eta  ]_{ (\lambda _1 + \lambda _2 ) \psi , (\lambda _1 + \lambda _2 )  \hat x}, $$ and this complete the proof of the statement.

\end{proof}

  \subsection{Normed space: Version 2}
In the usual case, we have that $d(x,y) =d(y,x)$ for any point $x,y \in X;$ on the other hand, in our intrinsic context, in general, we have that $$d(f(x),\pi^{-1} (y)) \ne d(f(y),\pi^{-1} (x)),$$ for every $x,y \in X.$ In particular, it holds
\begin{equation}\label{equationDIFFERENZA.22}
\begin{aligned}
d(f(y), \pi^{-1} (x)) -d(f(z), \pi^{-1} (x))  & \leq d(f(y),f(z)), \quad \forall x,y,z \in Y\\
d(f(x), \pi^{-1} (y)) -d(f(x), \pi^{-1} (z)) & \nleq  d(f(y),f(z)), \quad \mbox{for some } x,y,z \in Y.\\
\end{aligned}
\end{equation}

In fact, for any fixed $x,y,z \in Y,$ if we choose $a \in \pi^{-1} (x)$ such that
\begin{equation*}
d(f(z), \pi^{-1} (x)) =d(f(z), a),
\end{equation*}
we deduce that
\begin{equation*} \begin{aligned}
d(f(y), \pi^{-1} (x)) -d(f(z), \pi^{-1} (x)) & \leq  d(f(y),a)  -d(f(z), \pi^{-1} (x)) \\ & \leq   d(f(y),f(z)) +d(f(z), a ) -d(f(z), \pi^{-1} (x)) \\
& = d(f(y),f(z)),\\
\end{aligned}
\end{equation*}  i.e., the first inequality of \eqref{equationDIFFERENZA.22} holds. 
On the other hand, for the second inequality in \eqref{equationDIFFERENZA.22}, we give the following example. let $X\subset \R^2$ the set given by the three lines with vertex $(0,8), (8,8);$ $(1,4), (8,6)$ and $(0,3),(8,7)$ and the subset $Y$ of $\R^2$ defined as the line with vertex $(0,0)$ and $(8,0).$ On $X$ we consider the usual distance on $\R^2.$ Then, if we consider  a continuous section $f:Y \to X$ of the projection on the first component $\pi :X\to Y$ with $f(x)=f((1,0))=(1,4), f(y)=f((7,0))=(8,7)$ and $f(z)=f((6,0))=(8,6),$ it is easy to see that
\begin{equation*} \begin{aligned}
d(f(x), \pi^{-1} (y)) -d(f(x), \pi^{-1} (z)) & = \sqrt{\frac{5}{4}}, \\
 d(f(y),f(z)) &= 1,
\end{aligned}
\end{equation*} 
and so \begin{equation*} \begin{aligned}
d(f(x), \pi^{-1} (y)) -d(f(x), \pi^{-1} (z)) & \nleq  d(f(y),f(z)).\\
\end{aligned}
\end{equation*} 

Then, it is not trivial to consider the norm $|||.|||$ defined as 
   \begin{equation}\label{banachspace.18}
|||\phi |||_{ ILS_{\psi , \hat x}} := \|\phi\|_\infty + \{\phi\}_{\lambda \psi , \lambda \hat x},
\end{equation}
where $\|\phi \|_\infty := \sup_{y\in K} |\phi (y)|$ and 
\begin{equation*}
\{\phi\}_{\lambda \psi , \lambda \hat x}:= \sup_{y \in K}  d(\lambda \hat x ,(1/ \lambda \pi )^{-1} (y))  .
\end{equation*}
and to prove as above the following statement.
  \begin{theorem}\label{24theoremNormed.2}
   Let $\pi :\R^\kappa \to Y$ be a linear and quotient map with $Y$ a metric space. Assume also that $\psi :K \to \R^\kappa $ is an intrinsically $L$-Lipschitz section of $\pi$ with $K \subset Y $ compact, $L \geq 1$ and $\hat x \in X.$ Then, the set $ \LL$ endowed with  $|||\cdot |||_{ ILS_{\psi , \hat x}}$ as in \eqref{banachspace.18}   is a normed space.
       \end{theorem}

\begin{proof}
The proof follows in a similar way as in Theorem \ref{24theoremNormed.1}.
\end{proof}

%

\section{Level sets and extensions}\label{Level sets and extensions}

In this section we prove the following theorem.

\begin{theorem}[Extensions as level sets]\label{thm3.Rkappa}  
 Let $\pi:X\to Y \times \R^s$ be a quotient map 
between a  metric space $X$ and  a topological space $Y \times \R^s$. 

Assume that $X$ is geodesic and that there exist $k\geq 1,$ $\rho : X \times X \to \R$ k-biLipschitz equivalent to the distance of $X,$ and $\tau =(\tau_1,\dots , \tau _s) :X \to \R^s $ k-Lipschitz and $k$-biLipschitz on  fibers such that for all $\tau_0\in \R ^s$  \begin{enumerate}
\item  the set $\tau _1^{-1}(\tau_0)$ is an intrinsically $k$-Lipschitz graph of a section $\phi_{1,\tau_0}:Y\times \R \times \{ 0\}^{s-1} \to X;$ the set $\tau _2^{-1}(\tau_0)$ is an intrinsically $k$-Lipschitz graph of a section $\phi_{2,\tau_0}:Y\times \{ 0\} \times \R \times \{ 0\}^{s-2} \to X;$ \dots ,   the set $\tau _s^{-1}(\tau_0)$ is an intrinsically $k$-Lipschitz graph of a section $\phi_{s,\tau_0}:Y\times  \{ 0\}^{s-1} \times \R \to X;$
\item for all $x_0\in \tau _i^{-1} (\tau_0)$ the map $X\to \R, x \mapsto \delta _{i, \tau_0} (x) := \rho (x_0, \phi_{i, \tau_0}(\pi (x)))$ is $k$-Lipschitz on the set $\{|\tau _i| \leq \delta _{i, \tau_0} \}$.
\end{enumerate}
Let $Y'  \times \R^s \subset Y  \times \R^s$ a set and $L\geq1$.
Then for every  intrinsically $L$-Lipschitz section  $\phi : Y'  \times \R^s \to \pi^{-1}(Y'  \times \R^s) $  of $\pi|_{\pi^{-1}(Y'  \times \R^s )} :\pi^{-1}(Y'  \times \R^s) \to  Y'  \times \R^s$, there exists a  map $ f:X\to \R ^s$    that is $K$-Lipschitz and $K$-biLipschitz on  fibers, with $K= 2k(Lk+2)$, such that     \begin{equation}\label{equationluogozeri_intro}
\phi (Y'  \times \R^s )\subseteq f^{-1} (0).
\end{equation}
In particular, each `partially defined' intrinsically Lipschitz graph $\phi (Y'  \times \R^s )$ is a subset of a `globally defined' intrinsically Lipschitz graph $ f^{-1} (0)$.
 \end{theorem}

  We need to mention that there have been several earlier partial results on extensions of Lipschitz graphs, as for example in \cite{FSSC06},  \cite[Theorem~1.5]{Vittone20},  \cite[Proposition~4.8]{MR3194680}, in the  Heisenberg group with codimension larger than one; \cite[Proposition~3.4]{V12},  for the case of codimension one in the Heisenberg groups;  \cite[Theorem~4.1]{FS16},  for the case of codimension one in Carnot groups. Finally, for the general metric spaces the reader can see  \cite{ALPD20, LN05, O09}.

  \begin{proof}[Proof of Theorem~\ref{thm3.Rkappa} \noindent{\rm \bf(\ref{thm3.Rkappa}.i)}]
It is proved in \cite{DDLD21}.
\end{proof}

\medskip

  \begin{proof}[Proof of Theorem $\ref{thm3.Rkappa}$ \noindent{\rm \bf(\ref{thm3.Rkappa}.ii)}]  \noindent{\rm \bf Step 1.}   Fix $i=1,\dots , s$ and, for simplicity, we write $\tau ^{-1}, f_{x_0}$ instead of $\tau ^{-1}_i$ and $f_{x_0,i}.$
Fix $x_0 \in  \tau ^{-1} (\tau_0).$ We consider the map $f_{x_0} :X \to \R$ defined as
\begin{equation}\label{int}
f_{x_0}(x)=\left\{ 
\begin{array}{lcl}
2(\tau (x)-\tau (x_0) - \alpha \delta _{\tau _0} (x)) &   & \mbox{ if } \, \, |\tau (x)-\tau (x_0)| \leq 2\alpha \delta _{\tau _0}(x)  \\
\tau (x)-\tau (x_0) &    & \mbox{ if }\,\, \tau (x)-\tau (x_0) > 2\alpha \delta _{\tau _0}(x)  \\
3( \tau (x)-\tau (x_0)) &    & \mbox{ if }\,\, \tau (x)-\tau (x_0) < -2\alpha \delta _{\tau _0}(x)  \\
\end{array}
\right.
\end{equation}
where $\alpha := kL+1.$  
We prove that $f_{x_0}$ satisfies the following properties:
\begin{description}
\item[$(i)$]  $f_{x_0}$ is $K$-Lipschitz;
\item[$(ii)$]  $f_{x_0}(x_0)=0;$
\item[$(iii)$]  $f_{x_0}$ is biLipschitz on fibers. 
\end{description}
where $K= \max\{ 3k, 2k+2\alpha k \}=2k+2\alpha k$ because $\alpha >1.$ The property $(i)$ derives using that $\tau , \delta _{\tau _0}$ are both Lipschitz and $X$ is a geodesic space. On the other hand, $(ii)$ is true noting that  $\delta _{\tau _0}(x_0)= \rho (x_0, \phi_{\tau_0}(\pi (x_0)))=0$ because $x_0 \in \phi _{\tau_0} (Y).$

Finally, for any $x,x' \in \pi ^{-1} (y) $ we have that $\rho (x_0, \phi_{\tau_0}(\pi (x)))=\rho (x_0, \phi_{\tau_0}(\pi (x')))$  and so $f$ is biLipschitz on fibers because $\tau$ is so too.  
Hence $(iii)$ holds.

At this point, we consider the map $f:X \to \R$ given by
\begin{equation*}
f(x) := \sup _{x_0 \in \phi (Y)} f_{x_0} (x), \quad \forall x \in X,
\end{equation*}
and we want to prove that it is the map we are looking for. The Lipschitz properties are true recall that the function $\delta _{x_0} $ is constant on the fibers. Consequently, the only non trivial fact to show is \eqref{equationluogozeri_intro}. 
Fix $\bar x_0 \in \phi (Y')$.
By $(ii)$ we have that $f_{ \bar x_0}( \bar x_0)=0$ and so it is sufficient  to prove that $f_{x_0}(\bar x_0)\leq 0$ for $x_0\in \phi (Y').$ Let $x_0\in \phi (Y').$  Then  using in addition that $\tau$ is $k$-Lipschitz, and that $\phi$ is  intrinsically $L$-Lipschitz, we have
\begin{equation*}
|\tau (\bar x_0) -\tau( x_0)| \leq k d(\bar x_0 , x_0) \leq Lk d(x_0, \pi^{-1} (\pi(\bar x_0)) ) \leq  Lk d(x_0,  \phi_{\tau_0}(\pi (\bar x_0)) ) <\alpha \delta _{\tau _0} (\bar x _0),
\end{equation*}
and so
\begin{equation*}
f_{x_0}(\bar x_0) = 2(\tau (\bar x_0)-\tau (x_0) - \alpha \delta _{\tau _0} (\bar x_0)) <0,
\end{equation*}
i.e., \eqref{equationluogozeri_intro} holds.

 \noindent{\rm \bf Step 2.} We consider $f=(f_1,\dots , f_s)$ where each $f_i$ is given by  $\tau _i^{-1}.$ Roughly speaking, here the problem is that when we put together $(f_1,\dots , f_s)$ using $\tau _1^{-1},\dots , \tau _s^{-1} $ then $f$ can not Lipschitz. But now $f$ is Lipschitz thanks to the construction of $Y\times \R^s$ and in particular of $\phi_{1,\tau_0}:Y\times \R \times \{ 0\}^{s-1} \to X; \dots ;\phi_{s,\tau_0}:Y \times \{ 0\}^{s-1}\times \R \to X .$ 
   \end{proof}

%
%

 \bibliographystyle{alpha}
\bibliography{DDLD}

\newcommand{\etalchar}[1]{$^{#1}$}
\begin{thebibliography}{FSSC03b}

\bibitem[ABB19]{ABB}
Andrei Agrachev, Davide Barilari, and Ugo Boscain.
\newblock A comprehensive introduction to sub-{R}iemannian geometry.
\newblock {\em Cambridge Studies in Advanced Mathematics, Cambridge Univ.
  Press}, 181:762, 2019.

\bibitem[AGS14a]{AGS08.2}
L.~Ambrosio, N.~Gigli, and G.~Savar\'e.
\newblock Calculus and heat flow in metric measure spaces and applications to
  spaces with {R}icci bounds from below.
\newblock {\em Invent. Math.}, 195(2):289--391, 2014.

\bibitem[AGS14b]{AGS08.3}
L.~Ambrosio, N.~Gigli, and G.~Savar\'e.
\newblock Metric measure spaces with {R}iemannian {R}icci curvature bounded
  from below.
\newblock {\em Duke Math. J.}, 163(7):1405--1490, 2014.

\bibitem[AGS15]{AGS08.4}
L.~Ambrosio, N.~Gigli, and G.~Savar\'e.
\newblock Bakry-\'{E}mery curvature-dimension condition and {R}iemannian
  {R}icci curvature bounds.
\newblock {\em Ann. Probab.}, 43(1):339--404, 2015.

\bibitem[AK00]{AmbrosioKirchheimRect}
Luigi Ambrosio and Bernd Kirchheim.
\newblock Rectifiable sets in metric and {B}anach spaces.
\newblock {\em Math. Ann.}, 318(3):527--555, 2000.

\bibitem[AM22a]{AM2022.1}
G.~Antonelli and A.~Merlo.
\newblock On rectifiable measures in {C}arnot groups: existence of density.
\newblock {\em Accepted in Journal of Geometric Analysis}, 2022.

\bibitem[AM22b]{AM2022.2}
G.~Antonelli and A.~Merlo.
\newblock On rectifiable measures in {C}arnot groups: {M}arstrand--{M}attila
  rectifiability criterion.
\newblock {\em Accepted in Journal of Functional Analysis}, 2022.

\bibitem[AP20]{ALPD20}
Luigi Ambrosio and Daniele Puglisi.
\newblock Linear extension operators between spaces of {L}ipschitz maps and
  optimal transport.
\newblock {\em Journal f\"ur die reine und angewandte Mathematik (Crelles
  Journal)}, 2020(764):1--21, 2020.

\bibitem[Bat21]{DB22}
D.~Bate.
\newblock Characterising rectifiable metric spaces using tangent spaces.
\newblock {\em Accepted, Invent. math.}, 2021.

\bibitem[BJL{\etalchar{+}}99]{MR1736929}
S.~Bates, W.~B. Johnson, J.~Lindenstrauss, D.~Preiss, and G.~Schechtman.
\newblock Affine approximation of {L}ipschitz functions and nonlinear
  quotients.
\newblock {\em Geom. Funct. Anal.}, 9(6):1092--1127, 1999.

\bibitem[BLU07]{BLU}
A.~Bonfiglioli, E.~Lanconelli, and F.~Uguzzoni.
\newblock {\em Stratified {L}ie groups and potential theory for their
  sub-{L}aplacians}.
\newblock Springer Monographs in Mathematics. Springer, Berlin, 2007.

\bibitem[CDPT07]{CDPT}
Luca Capogna, Donatella Danielli, Scott~D. Pauls, and Jeremy~T. Tyson.
\newblock {\em An introduction to the {H}eisenberg group and the
  sub-{R}iemannian isoperimetric problem}, volume 259 of {\em Progress in
  Mathematics}.
\newblock Birkh\"{a}user Verlag, Basel, 2007.

\bibitem[Che99]{C99}
J.~Cheeger.
\newblock Differentiability of {L}ipschitz functions on metric measure spaces.
\newblock {\em Geom. Funct. Anal.}, 9(8):428--517, 1999.

\bibitem[CP06]{MR2247905}
D.R. Cole and S.D. Pauls.
\newblock {$C^1$} hypersurfaces of the {H}eisenberg group are
  {$N$}-rectifiable.
\newblock {\em Houston J. Math.}, 32(3):713--724, 2006.

\bibitem[DD22a]{D22.31may}
Daniela Di~Donato.
\newblock Intrinsic {C}heeger energy for the intrinsically {L}ipschitz
  constants.
\newblock {\em preprint}, 2022.

\bibitem[DD22b]{D22.NOV26A}
Daniela Di~Donato.
\newblock The intrinsic {H}opf-{L}ax semigroup vs. the intrinsic slope.
\newblock {\em preprint}, 2022.

\bibitem[DD22c]{D22.1}
Daniela Di~Donato.
\newblock Intrinsically {H}\"older sections in metric spaces.
\newblock {\em preprint}, 2022.

\bibitem[DD22d]{D22.3}
Daniela Di~Donato.
\newblock Intrinsically {L}ipschitz graphs on semidirect products of groups.
\newblock {\em preprint}, 2022.

\bibitem[DD22e]{D22.NOV26B}
Daniela Di~Donato.
\newblock Non-symmetric intrinsic {H}opf-{L}ax semigroup vs. intrinsic
  {L}agrangian.
\newblock {\em preprint}, 2022.

\bibitem[DDLD22]{DDLD21}
Daniela Di~Donato and Enrico Le~Donne.
\newblock Intrinsically {L}ipschitz sections and applications to metric groups.
\newblock {\em preprint}, 2022.

\bibitem[DS91]{DS91a}
G.~David and S.~Semmes.
\newblock Singular integrals and rectifiable sets in {$R^n$}: Beyond lipschitz
  graphs, ast\'erisque.
\newblock (193):171--279, 1991.

\bibitem[DS93]{DS91b}
G.~David and S.~Semmes.
\newblock Analysis of and on uniformly rectifiable sets.
\newblock {\em Mathematical Surveys and Monographs}, 38, 1993.

\bibitem[Fed69]{Federer}
Herbert Federer.
\newblock {\em Geometric measure theory}.
\newblock Die Grundlehren der mathematischen Wissenschaften, Band 153.
  Springer-Verlag New York Inc., New York, 1969.

\bibitem[FS16]{FS16}
Bruno Franchi and Raul~Paolo Serapioni.
\newblock Intrinsic {L}ipschitz graphs within {C}arnot groups.
\newblock {\em J. Geom. Anal.}, 26(3):1946--1994, 2016.

\bibitem[FSS22]{Savar2}
M.~Fornasier, G.~Savar\'e, and G.E. Sodini.
\newblock Density of subalgebras of lipschitz functions in metric sobolev
  spaces and applications to wasserstein sobolev spaces.
\newblock {\em preprint}, 2022.

\bibitem[FSSC01]{FSSC}
B.~Franchi, R.~Serapioni, and F.~Serra~Cassano.
\newblock Rectifiability and perimeter in the {H}eisenberg group.
\newblock {\em Math. Ann.}, 321(3):479--531, 2001.

\bibitem[FSSC03a]{MR2032504}
B.~Franchi, R.~Serapioni, and F.~Serra~Cassano.
\newblock Regular hypersurfaces, intrinsic perimeter and implicit function
  theorem in {C}arnot groups.
\newblock {\em Comm. Anal. Geom.}, 11(5):909--944, 2003.

\bibitem[FSSC03b]{FSSC03}
Bruno Franchi, Raul Serapioni, and Francesco Serra~Cassano.
\newblock On the structure of finite perimeter sets in step 2 {C}arnot groups.
\newblock {\em The Journal of Geometric Analysis}, 13(3):421--466, 2003.

\bibitem[FSSC06]{FSSC06}
Bruno Franchi, Raul Serapioni, and Francesco Serra~Cassano.
\newblock Intrinsic {L}ipschitz graphs in {H}eisenberg groups.
\newblock {\em J. Nonlinear Convex Anal.}, 7(3):423--441, 2006.

\bibitem[Kei04]{K04}
S.~Keith.
\newblock A differentiable structure for metric measure spaces.
\newblock {\em Adv. Math. 183}, pages 271--315, 2004.

\bibitem[KM16]{KM16}
B.~Kleiner and J.M. Mackay.
\newblock Differentiable structures on metric measure spaces: a primer.
\newblock {\em Ann. Sc. Norm. Super. Pisa Cl. Sci. (5) Vol. XVI}, pages 41--64,
  2016.

\bibitem[LN05]{LN05}
J.R. Lee and A.~Naor.
\newblock Extending lipschitz functions via random metric partitions.
\newblock {\em Invent. Math.}, 160(1):59--95, 2005.

\bibitem[LV09]{MR2480619}
John Lott and C\'{e}dric Villani.
\newblock Ricci curvature for metric-measure spaces via optimal transport.
\newblock {\em Ann. of Math. (2)}, 169(3):903--991, 2009.

\bibitem[Mag04]{MAGNANI1}
V.~Magnani.
\newblock Unrectifiability and rigidity in stratified groups.
\newblock {\em Arch. Math.}, 83(6):568--576, 2004.

\bibitem[Mon14]{MR3194680}
Roberto Monti.
\newblock Lipschitz approximation of {$\Bbb{H}$}-perimeter minimizing
  boundaries.
\newblock {\em Calc. Var. Partial Differential Equations}, 50(1-2):171--198,
  2014.

\bibitem[NY18]{NY18}
Assaf Naor and Robert Young.
\newblock Vertical perimeter versus horizontal perimeter.
\newblock {\em Ann. of Math., (2) 188}, (1):171--279, 2018.

\bibitem[Oht09]{O09}
S.I. Ohta.
\newblock Extending {L}ipschitz and {H}\"older maps between metric spaces.
\newblock {\em Positivity 13}, 2:407--425, 2009.

\bibitem[Pau04]{MR2048183}
S.D. Pauls.
\newblock A notion of rectifiability modeled on {C}arnot groups.
\newblock {\em Indiana Univ. Math. J.}, 53(1):49--81, 2004.

\bibitem[Sav22]{Savar1}
G.~Savar\'e.
\newblock Sobolev spaces in extended metric-measure spaces.
\newblock {\em In New Trends on Analysis and Geometry in Metric Spaces,
  Springer}, pages 117--276, 2022.

\bibitem[SC16]{SC16}
Francesco Serra~Cassano.
\newblock Some topics of geometric measure theory in {C}arnot groups.
\newblock In {\em Geometry, analysis and dynamics on sub-{R}iemannian
  manifolds. {V}ol. 1}, EMS Ser. Lect. Math., pages 1--121. Eur. Math. Soc.,
  Z\"urich, 2016.

\bibitem[Stu06]{S06}
K.T. Sturm.
\newblock On the geometry of metric measure spaces. {I}, {II}.
\newblock {\em Acta Math.}, 196(1):65-- 131 and 133--177, 2006.

\bibitem[Vil09]{MR2459454}
C\'{e}dric Villani.
\newblock {\em Optimal transport, Old and new}, volume 338 of {\em Grundlehren
  der Mathematischen Wissenschaften [Fundamental Principles of Mathematical
  Sciences]}.
\newblock Springer-Verlag, Berlin, 2009.

\bibitem[Vit]{Vittone20}
Davide Vittone.
\newblock Lipschitz graphs and currents in {H}eisenberg groups.
\newblock Forum of Mathematics, Sigma, 10, E6, 2022.

\bibitem[Vit12]{V12}
Davide Vittone.
\newblock Lipschitz surfaces, perimeter and trace theorems for {BV} functions
  in {C}arnot-{C}arath\'{e}odory spaces.
\newblock {\em Ann. Sc. Norm. Super. Pisa Cl. Sci. (5)}, 11(4):939--998, 2012.

\bibitem[VN88]{Berestovski}
Berestovskii Valerii~Nikolaevich.
\newblock Homogeneous manifolds with intrinsic metric.
\newblock {\em Sib Math J}, I(29):887--897, 1988.

\end{thebibliography}

\end{document}